\documentclass[a4paper,12pt,reqno]{article}
\usepackage[latin1]{inputenc}
\usepackage[T1]{fontenc}
\usepackage{lmodern}
\usepackage[english]{babel}
\usepackage{amsmath}
\usepackage{amssymb}
\usepackage{amsthm}
\usepackage{bm}
\usepackage[pagestyles]{titlesec}\usepackage[pdfpagemode=UseNone,pdfstartview=FitH]{hyperref}
\renewpagestyle{plain}{\setfoot[\thepage][][]{}{}{\thepage}}
\theoremstyle{plain}
\newtheorem{theorem}{Theorem}[section]

\theoremstyle{definition}
\newtheorem{definition}{Definition}[section]
\newtheorem{example}{Example}[section]
\newtheorem{remark}{Remark}[section]
\numberwithin{equation}{section}
\allowdisplaybreaks[1]
  
\newcommand*{\Zset}{\mathbb{Z}}

\newcommand*{\Cset}{\mathbb{C}}  
\begin{document}
\title{\textbf{Some properties of row-adjusted meet and join matrices}}
\author{}
\date{7.9.2011}
\maketitle
\begin{center}
\textsc{Mika Mattila$^*$ and Pentti Haukkanen}\\
School of Information Sciences\\
FI-33014 University of Tampere, Finland\\[5mm]

\end{center}
{\bf Abstract} Let $(P,\preceq)$ be a lattice, $S$ a finite subset of $P$ and $f_1,f_2,\ldots,f_n$ complex-valued functions on $P$. We define row-adjusted meet and join matrices on $S$ by $(S)_{f_1,\ldots,f_n}=(f_i(x_i\wedge x_j))$ and $[S]_{f_1,\ldots,f_n}=(f_i(x_i\vee x_j))$. In this paper we determine the structure of the matrix $(S)_{f_1,\ldots,f_n}$ in general case and in the case when the set $S$ is meet closed we give bounds for $\text{rank} (S)_{f_1,\ldots,f_n}$ and present expressions for $\det (S)_{f_1,\ldots,f_n}$ and $(S)_{f_1,\ldots,f_n}^{-1}$. The same is carried out dually for row-adjusted join matrix of a join closed set $S$.
\\[5mm]
\emph{$^\ast$Corresponding\ author.}\ \textit{Tel.:}\ +358\ 31\ 3551\ 7581,\ \textit{fax:}\ +358\ 31\ 3551\ 6157\\
\emph{E-mail addresses:}\ mika.mattila@uta.fi\ (M. Mattila),\\\hspace*{32mm} pentti.haukkanen@uta.fi\ (P. Haukkanen)

\newpage

\section{Introduction}

In 1876 Smith \cite{S} presented a formula for the determinant of the $n\times n$ matrix $((i,j))$, having the greatest common divisor of $i$ and $j$ as its $ij$ element. During the 20th century many other results concerning matrices with similar structure were published, see for example \cite{HWS, LS, W}. In 1989 Beslin and Ligh \cite{BeL} introduced the concept of a GCD matrix on a set $S$, where $S=\{x_1,x_2,\ldots,x_n\}\subset\Zset^+$ with $x_1<x_2<\cdots<x_n$ and the GCD matrix $(S)$ has $(x_i,x_j)$ as its $ij$ entry. Since then numerous publications have appeared in order to universalize the concept of GCD matrix. For example, Haukkanen \cite{H2} and Luque \cite{Lu} consider the determinants of multidimensional generalizations of GCD matrices and Hong, Zhou and Zhao \cite{Ho} study power GCD matrices for a unique factorization domain.

Poset theoretic generalizations of GCD matrices were first introduced by Lindström \cite{L}  and Wilf \cite{Wi}. In these generalizations $(P,\preceq)$ is a poset, $f$ is a function $P\to \Cset$, $S=\{x_1,x_2,\ldots,x_n\}\subset P$, $x_i\preceq x_j\Rightarrow i\leq j$ and $(S)_f$ is an $n\times n$ matrix with $f(x_i\wedge x_j)$ as its $ij$ element. These matrices are referred to as meet matrices. The papers by Lindström \cite{L} and Wilf \cite{Wi} arose from needs for combinatorics and became possible since Rota \cite{R} had previously developed his famous theory on Möbius functions. Rajarama Bhat \cite{RB} and Haukkanen \cite{H} were the first to investigate meet matrices systematically,  presenting many important properties of ordinary GCD matrices in terms of meet matrices. In \cite{KH} Korkee and Haukkanen define and study the join matrix $[S]_f$ of the set $S$ with respect to $f$, where $f(x_i\vee x_j)$ is the $ij$ element of the matrix $[S]_f$. 

During the last ten years the concept of meet matrix has been generalized even further in many different ways. Korkee \cite{K} studies the properties of a matrix $M_{S,f}^{\alpha,\beta,\gamma,\delta}$, which yields both the matrix $(S)_f$ and $[S]_f$ as its special case. A totally different approach is taken by Altinisik, Tuglu and Haukkanen in \cite{ATH}, when they define meet and join matrices on two subsets $X$ and $Y$ of $P$. A further idea of generalization is presented by Bege \cite{B} as he studies yet another GCD related matrix $(F(i,(i,j)))$, where $F(m,n)$ is an arithmetical function of two variables. For present purposes it is convenient to use a slightly different notation. For every $i\in\Zset^+$ we define an arithmetical function $f_i$ of one variable by
\begin{equation}\label{index}
f_i(m)=F(i,m)\quad\text{for all}\ m\in\Zset^+.
\end{equation}
With this notation Bege's matrix takes the form
\begin{equation}\label{matrix}
\left[ \begin{array}{cccc}
f_1((1,1)) & f_1((1,2)) & \cdots & f_1((1,n)) \\
f_2((2,1)) & f_2((2,2)) & \cdots & f_2((2,n)) \\
\vdots & \vdots & \ddots & \vdots \\
f_n((n,1)) & f_n((n,2)) & \cdots & f_n((n,n)) \\
\end{array} \right].
\end{equation}
In order to distinguish between this and the numerous other generalizations of GCD matrices, this matrix is referred to as the \emph{row-adjusted GCD matrix of the set} $\{1,2,\ldots,n\}$. This notation also enables us to define row-adjusted meet and join matrices.

\begin{definition}
Let $(P,\preceq)$ be a lattice, $S=\{x_1,x_2,\ldots,x_n\}$ be a finite subset of $P$ with $x_i\preceq x_j\Rightarrow i\leq j$ and $f_1,f_2,\ldots,f_n$ be complex-valued functions on $P$. The row-adjusted meet matrix of the set $S$ is the $n\times n$ matrix $(S)_{f_1,\ldots,f_n}$, which has $(f_i(x_i\wedge x_j))$ as its $ij$ element. Similarly, the row-adjusted join matrix $[S]_{f_1,\ldots,f_n}$ has $(f_i(x_i\vee x_j))$ as its $ij$ element.
\end{definition}

More explicitly,
\begin{equation}
(S)_{f_1,\ldots,f_n}=\left[ \begin{array}{cccc}
f_1(x_1\wedge x_1) & f_1(x_1\wedge x_2) & \cdots & f_1(x_1\wedge x_n) \\
f_2(x_2\wedge x_1) & f_2(x_2\wedge x_2) & \cdots & f_2(x_2\wedge x_n) \\
\vdots & \vdots & \ddots & \vdots \\
f_n(x_n\wedge x_1) & f_n(x_n\wedge x_2) & \cdots & f_n(x_n\wedge x_n) \\
\end{array} \right]
\end{equation}
and
\begin{equation}
[S]_{f_1,\ldots,f_n}=\left[ \begin{array}{cccc}
f_1(x_1\vee x_1) & f_1(x_1\vee x_2) & \cdots & f_1(x_1\vee x_n) \\
f_2(x_2\vee x_1) & f_2(x_2\vee x_2) & \cdots & f_2(x_2\vee x_n) \\
\vdots & \vdots & \ddots & \vdots \\
f_n(x_n\vee x_1) & f_n(x_n\vee x_2) & \cdots & f_n(x_n\vee x_n) \\
\end{array} \right].
\end{equation}
It turns out that there are some results concerning the matrix $(S)_{f_1,\ldots,f_n}$ to be found in the literature by Lindström \cite{L} and Luque \cite{Lu}. When the notation is the same as defined in \eqref{index}, these results can easily be applied to Bege's matrix.

Unlike the ordinary meet and join matrices, the matrices $(S)_{f_1,\ldots,f_n}$ and $[S]_{f_1,\ldots,f_n}$ are usually not symmetric. There are also many other key properties of meet and join matrices that do not hold for row-adjusted meet and join matrices. Hence, neither the traditional methods of meet and join matrices works in the study of these row-adjusted matrices.

\begin{remark}
In the case when $f_1=f_2=\cdots=f_n=f$, we have $(S)_{f_1,\ldots,f_n}=(S)_f$ and $[S]_{f_1,\ldots,f_n}=[S]_f$.
\end{remark}

\begin{remark}\label{column-adjusted}
Taking the transpose of a row-adjusted meet or join matrix results in a \emph{column-adjusted} meet or join matrix. Therefore the results concerning row-adjusted meet and join matrices can easily be translated for column-adjusted meet and join matrices using this connection.
\end{remark}

At the end of his paper Bege \cite{B} presents an open problem regarding the structure and the determinant of the matrix $(F(i,(i,j)))$. It appears that the question about the determinant could be solved using Lindström's result in \cite{L}. In this paper we present a more systematic investigation of the structure of $(S)_{f_1,\ldots,f_n}$ and $[S]_{f_1,\ldots,f_n}$ in general case. Then by using this knowledge we are able to find a different proof for Lindström's determinant formula and also prove some other results concerning the rank and inverse of these matrices.

\section{Preliminaries}

Let $(P,\preceq)$ be a lattice, $S=\{x_1,x_2,\ldots,x_n\}$ a finite subset of $P$ and 
\[f_1,f_2,\ldots,f_n:P\to\Cset\]
complex-valued functions on $P$. We also assume that the elements of $S$ are distinct and arranged so that
\[
x_i\preceq x_j\Rightarrow i\leq j.
\]
The set $S$ is said to be \emph{meet closed} if $x\wedge y\in S$ for all $x,y\in S$. In other words, the structure $(S,\preceq)$ is a meet semilattice. The concept of \emph{join closed set} is defined dually.

Let $D=\{d_1,d_2,\ldots,d_m\}$ be another subset of $P$ containing all the elements $x_i\wedge x_j$, $i,j=1,2,\ldots,n$, and having its elements arranged so that
\[
d_i\preceq d_j\Rightarrow i\leq j.
\]
Now for every $i=1,2,\ldots,n$ we define the function $\Psi_{D,f_i}$ on $D$ inductively as
\begin{equation}
\Psi_{D,f_i}(d_k)=f_i(d_k)-\sum_{d_v\prec d_k}\Psi_{D,f_i}(d_v),
\label{eq:Psi1}
\end{equation}
or equivalently
\begin{equation}
f_i(d_k)=\sum_{d_v\preceq d_k}\Psi_{D,f_i}(d_v).
\label{eq:Psi2}
\end{equation}
Thus we have
\begin{equation}
\Psi_{D,f_i}(d_k)=\sum_{d_v\preceq d_k}f_i(d_v)\mu_D(d_v,d_k),
\label{eq:Psi3}
\end{equation}
where $\mu_D$ is the Möbius function of the poset $(D,\preceq)$, see \cite[Section IV.1]{Aig} and \cite[3.7.1 Proposition.]{St}.

Let $E_D$ be the $n\times m$ matrix defined as
\begin{equation}\label{eq:E}
(e_D)_{ij}=\left\{
 \begin{array}{cc}
    1 & \textrm{if }d_{j}\preceq x_{i}\textrm{,} \\
    0 & \textrm{otherwise.}
 \end{array}
\right.
\end{equation}
The matrix $E_D$ may be referred to as the incidence matrix of the set $D$ with respect to the set $S$ and the partial ordering $\preceq$.

Finally, we need another $n\times m$ matrix $\Upsilon=(\upsilon_{ij})$, where
\begin{equation}\label{eq:Phi}
\upsilon_{ij}=(e_D)_{ij}\Psi_{D,f_i}(d_j).
\end{equation}
In other words, if $\Xi$ is the $n\times m$ matrix having $\Psi_{D,f_i}(d_j)$ as its $ij$ element, then $\Upsilon=E_D\circ\Xi$, the Hadamard product of the matrices $E_D$ and $\Xi$.

\section{A structure theorem}

In this section we give a factorization of the matrix $(S)_{f_1,\ldots,f_n}$, which then enables us to derive formulas for the rank, the determinant and the inverse of the matrix $(S)_{f_1,\ldots,f_n}$.

\begin{theorem}\label{th:meet.fac}
We have
\begin{equation}\label{eq:meet.fac}
(S)_{f_1,\ldots,f_n}=\Upsilon E_D^T=(E_D\circ\Xi)E_D^T.
\end{equation}
\end{theorem}

\begin{proof}
By \eqref{eq:Psi2}, \eqref{eq:E} and \eqref{eq:Phi} the $ij$ element of $(S)_{f_1,\ldots,f_n}$ is
\begin{equation}
f(x_i\wedge x_j)=\sum_{d_v\preceq x_i\wedge x_j}\Psi_{D,f_i}(d_v)=\sum_{k=1}^m (e_D)_{ik}\Psi_{D,f_i}(e_D)_{jk},
\end{equation}
which is the $ij$ element of the matrix $\Upsilon E_D^T$.
\end{proof}

\begin{remark}
It is possible to define  row-adjusted meet and join matrices $(X,Y)_{f_1,\ldots,f_n}$ and $[X,Y]_{f_1,\ldots,f_n}$ on two sets $X$ and $Y$ by $((X,Y)_{f_1,\ldots,f_n})_{ij}=f_i(x_i\wedge y_j)$ and $([X,Y]_{f_1,\ldots,f_n})_{ij}=f_i(x_i\vee y_j)$. It would be possible to generalize Theorem \ref{th:meet.fac} for these matrices, but the methods used in the proofs of the other theorems do not work in this general case.
\end{remark}

\begin{remark}\label{re:calc}
In the case when the set $S$ is meet closed Theorem \ref{th:meet.fac} also provides an effective way to calculate all the necessary values $\Psi_{S,f_i}(x_j)$ as follows. In this case $D=S$ and both $E_S$ and $\Upsilon$ are square matrices of size $n\times n$. Since $E_S$ is also invertible, from equation \eqref{eq:meet.fac} we obtain
\begin{equation}
\Upsilon=(S)_{f_1,\ldots,f_n} (E_S^T)^{-1},
\end{equation}
which gives the values of $\Psi_{S,f_i}(x_j)$. Here the matrix $E_S^T$ is the matrix associated with the zeta function $\zeta_S$ of the set $S$ (see \cite[p. 139]{Aig}), and thus the matrix $(E_S^T)^{-1}$ is the matrix of the Möbius function of the set $S$ and has $\mu_S(x_i,x_j)$ as its $ij$ element.
\end{remark}

The following example gives a solution for the first part of Bege's problem.

\begin{example}
The row-adjusted GCD matrix of the set $S=\{1,2,\ldots,n\}$ is the product of the matrices $\Upsilon=(\upsilon_{ij})$ and $E_S^T$, where
\begin{equation}
(e_S)_{ij}=\left\{
 \begin{array}{cc}
    1 & \textrm{if }j\,|\,i\textrm{,} \\
    0 & \textrm{otherwise}
 \end{array}
\right.
\end{equation}
and
\begin{equation}
\upsilon_{ij}=(e_S)_{ij}\Psi_{S,f_i}(j)=(e_S)_{ij}\sum_{k\,|\,j}f_i(k)\mu\left(\frac{j}{k}\right)=(e_S)_{ij}(f_i\ast\mu)(j),
\end{equation}
where $\ast$ is the Dirichlet convolution and $\mu$ is the number-theoretic Möbius function. It should be noted that here the notation $F(i,k)=f_i(k)$ is not only convenient but also enables the use of the Dirichlet convolution.

\end{example}

\section{Rank estimations}

In this section we derive bounds for $\text{rank}\,(S)_{f_1,\ldots,f_n}$ in the case when the set $S$ is meet closed. The rank of  meet and join matrices or even GCD and LCM matrices has not been studied earlier in the literature. 

\begin{theorem}\label{th:rank}
Let $S$ be a meet closed set and let $k$ be the number of indices $i$ with $\Psi_{D,f_i}(x_i)= 0$. Then the following properties hold.
\begin{enumerate}
\item $\mathrm{rank}\,(S)_{f_1,\ldots,f_n}=0$ iff $f_i(x_i\wedge x_j)=0$ for all $i,j=1,\ldots,n$.
\item If $k=0$, then $\mathrm{rank}\,(S)_{f_1,\ldots,f_n}=n$.
\item If $k>0$, then 
\begin{equation}
n-k\leq\mathrm{rank}\,(S)_{f_1,\ldots,f_n}\leq n-1.
\end{equation}
\end{enumerate}
\end{theorem}

\begin{proof}
\begin{enumerate}
\item Follows trivially.
\item
By Theorem \ref{th:meet.fac} we have
\begin{equation}
\mathrm{rank}\,(S)_{f_1,\ldots,f_n}=\mathrm{rank}\,\left(\Upsilon E_S^T\right).
\end{equation}
Since in this case the matrices $\Upsilon$ and $E_S$ are both triangular square matrices with full rank, the claim follows immediately.
\item
Since multiplying with the invertible matrix $E_S^T$ does not change the rank, we have 
\begin{equation}
\mathrm{rank}\,(S)_{f_1,\ldots,f_n}=\mathrm{rank}\,\Upsilon.
\end{equation}
To obtain the latter inequality we only need to note that since at least one of the diagonal elements of $\Upsilon$ equals zero, the rows of $\Upsilon$ cannot be linearly independent and thereby $\Upsilon$ cannot have a full rank. On the other hand, the $n-k$ rows with nonzero diagonal elements constitute a linearly independent set, from which we obtain the first inequality.
\end{enumerate}
\end{proof}

In the case when the set $S$ is meet closed and $f_1=\cdots=f_n=f$ (that is in the case of ordinary meet matrix) the question of the rank becomes trivial. Namely, the matrix $(S)_f$ can be written as
\begin{equation}
(S)_f=E_S\Lambda E_S^T,
\end{equation}
where $\Lambda=\mathrm{diag}\,(\Psi_{S,f}(x_1),\Psi_{S,f}(x_2),\ldots,\Psi_{S,f}(x_n))$, see \cite[Theorem 3.1.]{ATH}. Now by the same argument as in the proof of Theorem \ref{th:rank} we have
\begin{equation}
\mathrm{rank}\,(S)_f=\mathrm{rank}\,\Lambda=n-k.
\end{equation}
The following two examples show that the bounds in Theorem \ref{th:rank} are the best possible under these assumptions. They also show that a large value of $k$ may indicate a large decline of the rank of the row-adjusted meet matrix, but not necessarily.

\begin{example}
Let $x_1=x_i\wedge x_j$ for all $i,j=1,\ldots,n$, which implies that $x_1$ is the smallest element of $S$ and the set $S\backslash \{x_1\}$ is an antichain. Now the set $S$ is clearly meet closed, and for every $i=2,\ldots,n$ we have
\begin{equation}
\Psi_{S,f_i}(x_i)=f_i(x_i)-f_i(x_1).
\end{equation}
If $i>1$ and we set $f_i(x_i)=f_i(x_1)$, then the $i$th column of $\Upsilon$ becomes the zero vector and thus for every $i>1$ we may reduce the rank of the matrix $(S)_{f_1,\ldots,f_n}$ by one. Therefore if the first diagonal element of $\Upsilon$ is not zero, then $\mathrm{rank}\,(S)_{f_1,\ldots,f_n}=n-k$.
\end{example}

\begin{example}
Let $(P,\preceq)=\mathcal{N}_5$ and $S=P$ as shown in Figure \ref{fig: kuva}. Let
\begin{equation}
f_2(x_2)=f_3(x_1)=f_3(x_3)=f_4(x_3)=f_4(x_4)=f_5(x_4)=f_5(x_5)=1
\end{equation}
and $f_i(x_j)=0$ otherwise. Simple calculations show that 
$\Psi_{S,f_2}(x_2)=1\neq 0$,
\begin{equation}
\Psi_{S,f_1}(x_1)=\Psi_{S,f_3}(x_3)=\Psi_{S,f_4}(x_4)=\Psi_{S,f_5}(x_5)=0,
\end{equation}
and thereby $k=4$. But on the other hand we have
\begin{equation}
(S)_{f_1,\ldots,f_n}=\left[ \begin{array}{ccccc}
0 & 0 & 0 & 0 & 0 \\
0 & 1 & 0 & 0 & 1 \\
1 & 1 & 1 & 1 & 1 \\
0 & 0 & 1 & 1 & 1 \\
0 & 0 & 0 & 1 & 1
\end{array} \right],
\end{equation}
and clearly $\mathrm{rank}\,(S)_{f_1,\ldots,f_n}=4$.
\end{example}

\begin{figure}[ht]
\centering
\setlength{\unitlength}{0.7cm}
\begin{picture}(7,8)
\thicklines
\put(4,1){\line(-1,1){3}}
\put(4,1){\line(1,1){2}}
\put(6,3){\line(0,1){2}}
\put(1,4){\line(1,1){3}}
\put(6,5){\line(-1,1){2}}
\put(0.1,4){$x_2$}
\put(4,0.5){$x_1$}
\put(6.5,3){$x_3$}
\put(6.5,5){$x_4$}
\put(4,7.5){$x_5$}
\put(4,1){\circle*{0.2}}
\put(6,3){\circle*{0.2}}
\put(1,4){\circle*{0.2}}
\put(4,7){\circle*{0.2}}
\put(6,5){\circle*{0.2}}

\end{picture}
\caption{The lattice $\mathcal{N}_5$ and the choices of the elements of the set $S$.}
\label{fig: kuva}
\end{figure}
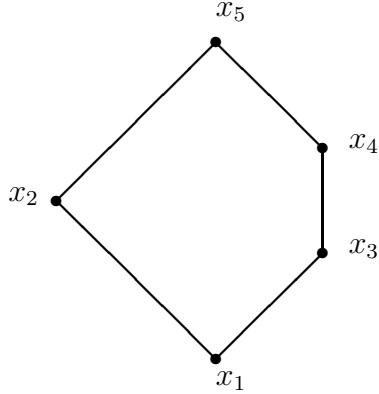

\section{Determinant formula}

In this section we present a determinant formula for the matrix $(S)_{f_1,\ldots,f_n}$ when the set $S$ is meet closed. This theorem is almost the same as that presented by Lindström \cite{L}. It is possible to use the Cauchy-Binet equality to obtain a determinant formula for $(S)_{f_1,\ldots,f_n}$ in general case. Since it is similar to the case of usual meet matrix, we do not present it here.

\begin{theorem}[\cite{L}, Theorem]\label{th:det}
If the set $S$ is meet closed, then
\begin{equation}\label{eq:det}
\det (S)_{f_1,\ldots,f_n}=\prod_{i=1}^n \Psi_{S,f_i}(x_i)=\prod_{i=1}^n
                 \sum_{x_j\preceq x_i}f_i(x_j)\mu_S(x_j,x_i).
\end{equation}
\end{theorem}
\begin{proof}
Since the set $S$ is meet closed, we have $D=S$. Then the matrix $E_S$ is a lower triangular square matrix having every main diagonal element equal to $1$. The matrix $\Upsilon$ is a lower triangular square matrix with $\Psi_{S,f_1}(x_1), \Psi_{S,f_2}(x_2),\ldots,\Psi_{S,f_n}(x_n)$ as diagonal elements. Thus $\det E_S=1$ and by Theorem \ref{th:meet.fac} we have
\begin{equation}
\det (S)_{f_1,\ldots,f_n}=\det \Upsilon=\prod_{i=1}^n \Psi_{S,f_i}(x_i).
\end{equation}
The second equality follows from \eqref{eq:Psi3}.

\end{proof}

\begin{remark}
The original theorem by Lindström \cite{L} is slightly more general since it does not require the assumption $x_i\preceq x_j\Rightarrow i\leq j$. As he states, the rows and columns of $(S)_{f_1,\ldots,f_n}$ can always be permuted in a way that does not change the determinant but makes the matrix $(S)_{f_1,\ldots,f_n}$ to fulfill this condition. 

\end{remark}

The following example gives a solution to the second part of Bege's problem.

\begin{example}
For the row-adjusted GCD matrix on the set $S=\{1,2,\ldots,n\}$ we have
\begin{equation}
\det(\{1,2,\ldots,n\})_{f_1,\ldots,f_n}=\prod_{i=1}^n \Psi_{S,f_i}(i)=\prod_{i=1}^n\sum_{j\,|\,i}f_i(j)\mu\left(\frac{i}{j}\right)=\prod_{i=1}^n(f_i\ast\mu)(i).
\end{equation}
\end{example}

\section{Inverse formula}

In this section we study the inverse of the matrix $(S)_{f_1,\ldots,f_n}$ when the set $S$ is meet closed. A formula for $(S)_{f_1,\ldots,f_n}^{-1}$ in general case could be obtained with the aid of meet matrices on two sets and the Cauchy-Binet equation. We do not, however, present the details here.

\begin{theorem}\label{th:inv}
If the set $S$ is meet closed, then the matrix $(S)_{f_1,\ldots,f_n}$ is invertible iff $\Psi_{S,f_i}(x_i)\neq 0$ for all $i=1,\ldots,n$. Furthermore, in this case the inverse of $(S)_{f_1,\ldots,f_n}$ is the $n\times n$ matrix $B=(b_{ij})$ with
\begin{equation}\label{eq:inv}
b_{ij}=\sum_{k=j}^n\mu_S(x_i,x_k)\theta_{kj},
\end{equation}
where the numbers $\theta_{jj},\theta_{j+1,j},\ldots,\theta_{nj}$ are defined recursively as
\begin{equation}\label{eq:theta}
\theta_{kj}=\left\{
 \begin{array}{cc}
    \frac{\displaystyle 1}{\displaystyle \Psi_{S,f_j}(x_j)} & \textrm{if }k=j\textrm{,} \\
    -\frac{\displaystyle 1}{\displaystyle \Psi_{S,f_k}(x_k)}{\displaystyle \sum_{u=j}^{k-1}e_{ku}\Psi_{S,f_k}(x_u)\theta_{uj}} & \textrm{if }k>j.
 \end{array}
\right.
\end{equation}
\end{theorem}

\begin{proof}
The first part follows directly from Theorem \ref{th:det}. To prove the second part we use Theorem \ref{th:meet.fac} and we obtain
\begin{equation}
(S)_{f_1,\ldots,f_n}^{-1}=(E_S^T)^{-1}\Upsilon^{-1}.
\end{equation}
In order to obtain the $ij$ element of the matrix $(S)_{f_1,\ldots,f_n}^{-1}$ we only have to ascertain the $i$th row of $(E_S^T)^{-1}$ and the $j$th column of $\Upsilon^{-1}$.
As stated in Remark \ref{re:calc}, the matrix $(E_S^T)^{-1}$ is the matrix associated with the Möbius function of the set $S$. Therefore its $i$th row is

\begin{equation}\label{eq:Möbius}
\left[0\ \ldots\ 0\ \underbrace{\mu_S(x_i,x_i)}_{=1}\ \mu_S(x_i,x_{i+1})\ \ldots\ \mu_S(x_i,x_n)\right].
\end{equation}

Now let $\Theta=(\theta_{ij})$ denote the inverse of $\Upsilon$. By multiplying the $j$th row of $\Upsilon$ with the $j$th column of $\Theta$, we obtain
\begin{equation}
\Psi_{S,f_j}(x_j)\theta_{jj}=1.
\label{eq:inv1}
\end{equation}
Further, the multiplication of the $k$th row of $\Upsilon$ and the $j$th column of $\Theta$ results in
\begin{equation}
\sum_{u=j}^ke_{ku}\Psi_{S,f_k}(x_u)\theta_{uj}=0.
\label{eq:inv2}
\end{equation}
Thus we obtain \eqref{eq:theta}, and \eqref{eq:inv} follows when we multiply the matrices $\Theta$ and $(E_S^T)^{-1}$.

\end{proof}

\section{Formulas for row-adjusted join matrices}

In this section the results presented in previous sections are translated for row-adjusted join matrices. The proofs of these dual theorems are omitted for the sake of brevity. Row-adjusted join matrices (or even row-adjusted LCM matrices) have not previously been studied in the literature. As stated in Remark \ref{column-adjusted}, the study of column-adjusted join matrices can easily be reverted to the study of row-adjusted join matrices via taking the transpose.

Let $D'=\{d_1',d_2',\ldots,d_{m'}'\}$ be a subset of $P$ containing all the elements $x_i\vee x_j$, $i,j=1,2,\ldots,n$, and having its elements arranged so that
\[
d_i'\preceq d_j'\Rightarrow i\leq j.
\]
For every $i=1,2,\ldots,n$ we define the function $\Psi_{D',f_i}'$ on $D'$ inductively as
\begin{equation}
\Psi_{D',f_i}'(d_k')=f_i(d_k')-\sum_{d_k'\prec d_v'}\Psi_{D',f_i}'(d_v'),
\label{eq:Psi1'}
\end{equation}
or equivalently
\begin{equation}
f_i(d_k')=\sum_{d_k'\preceq d_v'}\Psi_{D',f_i}'(d_v').
\label{eq:Psi2'}
\end{equation}
Thus we have
\begin{equation}
\Psi_{D',f_i}'(d_k')=\sum_{d_k'\preceq d_v'}f_i(d_v')\mu_{D'}(d_k',d_v'),
\label{eq:Psi3'}
\end{equation}
where $\mu_{D'}$ is the Möbius function of the poset $(D',\preceq)$, see \cite[3.7.2 Proposition.]{St}.

Let $E_{D'}'$ be the $n\times m'$ matrix defined as
\begin{equation}\label{eq:E'}
(e_{D'}')_{ij}=\left\{
 \begin{array}{cc}
    1 & \textrm{if }x_{i}\preceq d_{j}'\textrm{,} \\
    0 & \textrm{otherwise.}
 \end{array}
\right.
\end{equation}

Finally, let $\Upsilon'=(\upsilon_{ij}')$ be the $n\times m'$ matrix, where
\begin{equation}\label{eq:Phi'}
\upsilon_{ij}'=(e_{D'}')_{ij}\Psi_{D',f_i}'(d_j').
\end{equation}

\begin{theorem}\label{th:join1}
\begin{equation}
[S]_{f_1,\ldots,f_n}=\Upsilon'(E_{D'}')^T.
\label{eq:join1}
\end{equation}

\end{theorem}

\begin{theorem}\label{th:joinrank}
Let $S$ be a join closed set and let $k$ be the number of indices $i$ with $\Psi_{D',f_i}'(x_i)= 0$. Then the following properties hold.
\begin{enumerate}
\item $\mathrm{rank}\,[S]_{f_1,\ldots,f_n}=0$ iff $f_i(x_i\vee x_j)=0$ for all $i,j=1,\ldots,n$.
\item If $k=0$, then $\mathrm{rank}\,[S]_{f_1,\ldots,f_n}=n$.
\item If $k>0$, then 
\begin{equation}
n-k\leq\mathrm{rank}\,[S]_{f_1,\ldots,f_n}\leq n-1.
\end{equation}
\end{enumerate}
\end{theorem}

\begin{theorem}\label{th:joindet}
If the set $S$ is join closed, then
\begin{equation}\label{eq:joindet}
\det [S]_{f_1,\ldots,f_n}=\prod_{i=1}^n \Psi_{S,f_i}'(x_i)=\prod_{i=1}^n
                 \sum_{x_i\preceq x_j}f_i(x_j)\mu_S(x_i,x_j).
\end{equation}
\end{theorem}

\begin{theorem}\label{th:joininv}
If the set $S$ is join closed, then the matrix $[S]_{f_1,\ldots,f_n}$ is invertible iff $\Psi_{S,f_i}'(x_i)\neq 0$ for all $i=1,\ldots,n$. Furthermore, in this case the inverse of $[S]_{f_1,\ldots,f_n}$ is the $n\times n$ matrix $B'=(b_{ij}')$ with
\begin{equation}\label{eq:joininv}
b_{ij}'=\sum_{k=1}^j\mu_S(x_k,x_i)\theta_{kj}',
\end{equation}
where the numbers $\theta_{jj}',\theta_{j-1,j}',\ldots,\theta_{1j}'$ are defined recursively as
\begin{equation}\label{eq:theta'}
\theta_{kj}'=\left\{
 \begin{array}{cc}
    \frac{\displaystyle 1}{\displaystyle \Psi_{S,f_j}'(x_j)} & \textrm{if }k=j\textrm{,} \\
    -\frac{\displaystyle 1}{\displaystyle \Psi_{S,f_k}'(x_k)}{\displaystyle \sum_{u=k+1}^{j}e_{ku}'\Psi_{S,f_k}'(x_u)\theta_{uj}'} & \textrm{if }j>k.
 \end{array}
\right.
\end{equation}
\end{theorem}

\end{document}